\documentclass[runningheads]{llncs}
\usepackage{graphicx}
\usepackage{ epsfig, sectsty, graphicx, float, subfig, url,tabularx}
\usepackage{amsmath, amssymb}
\usepackage{algorithm, algorithmicx}
\usepackage[noend]{algpseudocode}
\usepackage{blindtext}
\usepackage{graphicx}
\usepackage{color}

\newcommand{\opt}{\mathrm{OPT}}

\sectionfont{\large} \subsectionfont{\normalsize}%

\newcommand{\mb}[1]{\ensuremath{\boldsymbol{#1}}}

\usepackage{footnote}
\makesavenoteenv{tabular}
\makesavenoteenv{table}

\begin{document}
\title{On the Power of Static Assignment Policies for Robust Facility Location Problems}
\titlerunning{Robust Facility Location Problems}
%
\author{Omar El Housni\inst{1} \and
Vineet Goyal\inst{2} \and
David Shmoys\inst{3} }
\authorrunning{El Housni et al.}
%
\institute{ORIE, Cornell Tech, New York, USA\\  
\email{oe46@cornell.edu}  \and
IEOR, Columbia University , New York, USA \\
\email{vg2277@columbia.edu} \and
ORIE, Cornell University, Ithaca, USA\\
\email{david.shmoys@cornell.edu}}
\maketitle              
\begin{abstract}
We consider a two-stage robust facility location problem on a metric under an uncertain demand. The decision-maker needs to decide on the (integral) units of supply for each facility in the first stage to satisfy an uncertain second-stage demand, such that the sum of first stage supply cost and the worst-case cost of satisfying the second-stage demand over all scenarios is minimized. The second-stage decisions are only assignment decisions without the possibility of adding  recourse supply capacity. This makes our model different from existing work on two-stage robust facility location and set covering problems. We consider an implicit model of uncertainty with an exponential number of demand scenarios specified by an upper bound $k$ on the number of second-stage clients. In an optimal solution, the second-stage assignment decisions depend on the scenario; surprisingly,  we show that  restricting to a fixed (static) fractional assignment for each potential  client irrespective of the scenario  gives us an $O(\log k/\log \log k)$-approximation for the problem. Moreover, the best such static assignment can be computed efficiently giving us the desired guarantee.


\end{abstract}
\section{Introduction}

We consider two-stage robust facility location problems under demand uncertainty where we are given a set of clients and a set of facilities in a common metric space. In the first stage, the decision-maker needs to select the (integral) units supply at each facility. The uncertain demand is then selected adversarially and needs to be satisfied by the existing supply with the minimum assignment cost in the second stage.  The goal is to determine the first-stage supply such that the sum of first-stage supply cost and the worst-case  assignment cost over all demand scenarios is minimized. Our problem is motivated by settings where the lead time to procure supply is large and obtaining additional units of supply in the second stage is not feasible. The uncertain second-stage demand must then be  satisfied by supply units from the first stage, a common constraint in many applications. This is a departure from existing work on two-stage robust facility location, network design and, more generally, robust covering problems that have been studied extensively in the literature~\cite{dhamdhere2005pay,FJMM07,anthony2010plant,gupta2014thresholded} where the second-stage decisions allow for ``adding more supply'' (more specifically adding more sets/facilities to satisfy the requirement). 



In this paper, we consider  an  implicit  model  of  uncertainty  with exponentially-many demand scenarios specified by an upper bound $k$ on the number of second-stage demand clients.  Since the number of scenarios is exponentially-many (specified compactly), we can not efficiently solve even the LP relaxation for the problem. In contrast, if the set of second-stage scenarios are explicitly specified (for the explicit scenario model, see for instance \cite{dhamdhere2005pay,anthony2010plant}), we can write a polynomially-sized LP relaxation with assignment decisions for each scenario. The main challenge then is related to obtaining an integral solution, which for the case of set covering and several network design problems can be reduced to deterministic versions (see, for instance, Dhamdhere et al. \cite{dhamdhere2005pay}). 
 
In contrast, in an implicit model of uncertainty (with possibly exponentially-many scenarios), one of the fundamental challenges is to even approximately solve the linear relaxation of the problem efficiently. The implicit model of uncertainty with an upper bound on number of uncertain second-stage clients or elements has been studied extensively in the literature. 
 Feige et al. \cite{FJMM07} show under a reasonable complexity assumption that it is hard to  solve the LP relaxation of a two-stage set covering problem within a factor better than $\Omega( \log n / \log \log n)$ under this implicit model of uncertainty. They also give an $O(\log^2n)$-approximation for the $0-1$ two-stage robust set-covering problem. Gupta et al.~\cite{gupta2014thresholded} give an improved $O(\log n)$-approximation for the set-covering problem, thereby matching the deterministic approximation guarantee. El Housni and Goyal~\cite{housni2018optimality} show that a static policy (that is, linear in the set of second-stage elements, also referred to as an affine policy) gives an $O(\log n/\log\log n)$-approximation for the two-stage LP, thereby matching the hardness lower bound for the fractional problem. Gupta et al. \cite{gupta2016robust} and Khandekar et al. \cite{khandekar2008two} present approximations for several network design problems under the implicit model of uncertainty. Although there is a large body of work in this direction, as we mentioned earlier, our problem is different from set covering, since there is no possibility of adding recourse capacity;  prior results do not  imply an approximation for our model.


\noindent
{\bf Other related work.}
Many variants of  facility location problems have been studied extensively in the literature, including both deterministic versions as well as variants that address demand uncertainty. We refer the reader to \cite{shmoys2000approximation} for a  review of the many variants of deterministic facility location problems. Among the models that address demand uncertainty in the facility location problems, in addition to robust  \cite{atamturk2007,charikar2001algorithms}, there are also stochastic \cite{gupta2004boosted,immorlica2004costs,swamy2006approximation,ravi2004hedging} and distributionally robust \cite{linhares2019approximation,basciftci2019distributionally,delage2010distributionally} models that have been studied extensively in the literature. In a stochastic model, there is a distribution on the second-stage demand scenarios and the goal is to minimize the total expected cost.
A distributionally robust model can be thought of as a hybrid between stochastic and robust where the second-stage distribution is selected adversarially from a pre-specified set and the goal is to minimize the worst-case expected cost. We refer the reader to the survey \cite{snyder2006facility} for an extensive review of facility location problems under uncertainty.

\subsection{Our Contributions} 

Let us begin with a formal problem definition. We are given a set of $n$ facilities $\cal F$ and $m$ clients $\cal C$ in a common metric $d$ where $d_{ij}$ denotes the distance between $i$ and $j$. 
For each facility $i \in {\cal F}$, there is a cost $c_i$ per unit of supply at $i$. 
The demand uncertainty is modeled by an implicit set of scenarios ${\cal C}^k$ that includes all subsets of clients $\cal C$ of size at most $k$.
The decision-maker needs to select an (integral) number of units of supply $x_i$ for each facility $i\in \cal F$ in the first-stage. An adversary observes the first stage decisions and selects a worst-case demand scenario $S \in {\cal C}^k$  that must be satisfied with the first-stage supply, where each client in the realized scenario needs one unit of supply. The goal is to minimize the sum of the first-stage supply cost and the worst-case  assignment cost over all second-stage demand scenarios. We refer to this problem as {\textit {soft-capacitated robust facility location}} \eqref{scrfl}. Typically in the literature, {\em soft-capacitated} refer to settings where violations of capacity upper bounds are allowed. The analogue here is that we can add any amount of supply in a facility without  upper bounds ($x_i \in \mathbb{Z_+})$ but we pay a per unit cost of supply. %

We consider a class of static assignment policies, where each of the $m$ clients has a static fractional assignment to facilities that is independent of the scenario, leading to a feasible second-stage solution for each demand scenario, while respecting supply capacities. Note this is a restriction, since the optimal second-stage assignment decisions are scenario-dependent in general. As a warm-up, we show that static assignment policies are optimal for the  uncapacitated case with unlimited supply at each open facility (i.e., there is a cost $c_i$ to open facility $i$ with unlimited supply). We refer to this problem as {\textit {uncapacitated robust facility location}} \eqref{urfl}. This is based on the intuition that each client can be assigned to the closest open facilities in an optimal solution in any scenario; this leads to  optimality of a static assignment policy for the LP relaxation (Theorem \ref{thm:static-opt}).

\begin{theorem} \label{thm:static-opt}
A static assignment policy is optimal for the linear relaxation of \eqref{urfl}. 
\end{theorem}



The optimality of static assignment is not true in general when the supply at facilities is constrained (or equivalently, there is a cost per unit of supply). The main contribution in this paper is to show that static assignment policies give an $O(\log k/\log \log k)$-approximation for the LP relaxation of \eqref{scrfl} (Theorem~\ref{thm:static:log}). We show this by constructing such a solution, starting from an optimal first-stage supply. The optimal static assignment policies can be computed efficiently by solving a compact LP.

\begin{theorem}\label{thm:static:log}
A static assignment policy  gives $O(\log k / \log \log k)$-approximation for the linear relaxation of \eqref{scrfl}.
\end{theorem}

Furthermore, the fractional supply in the first stage can be rounded to an integral supply using ideas similar to rounding algorithms for the deterministic facility location \cite{shmoys1997approximation}. In particular, the static assignment solution for the uncapacitated case can be rounded to give a $4$-approximation algorithm for \eqref{urfl}.  The static assignment solution for the soft-capacitated  case can be rounded within a constant factor, which results in  an $O(\log k / \log \log k)$-approximation algorithm for \eqref{scrfl}. We would like to note that while the fractional assignment is static in our approximate LP solution, our integral assignment for any client in the second-stage  depends on the other demand clients in the scenario; thereby, making our static assignment policy adaptive in implementation.

\section{Warm-up: uncapacitated robust facility location  }

\subsection{Problem formulation}

In this section, we consider the {\textit{uncapacitated robust facility location problem}} ({\sf URFL}) where for each $i \in {\cal F}$, there is a cost $c_i$ to open facility $i$ with unlimited supply. The problem can be stated as the following integer program, where each binary variable $x_i, i \in {\cal F}$ indicates if  facility $i$ is opened and each $y_{ij}^S, i \in {\cal F}, j \in S, S \in {\cal C}^k$ indicates the assignment of client $j$  to facility $i$ in scenario $S$.

\begin{equation}\label{urfl} \tag{\sf URFL}
\begin{aligned}
\min \; &  \sum_{i \in {\cal F}} c_i x_i +\max_{ S \in {\cal C}^k} \sum_{i \in {\cal F}} \sum_{ j \in {S}}  d_{ij} y_{ij}^S &\\
\text{s.t.    } & \sum_{i \in {\cal F}} y_{ij}^S \geq 1, \qquad &   \forall S \in {\cal C}^k,\forall j \in {S},     \\
& x_i \geq y_{ij}^S, \qquad  & \forall i  \in {\cal F},   \forall S \in {\cal C}^k,\forall j \in {S}, \\
& x_i \in \{0,1\}, \; y_{ij}^S \geq  0, \qquad & \forall i  \in {\cal F},   \forall S \in {\cal C}^k, \forall j \in {S}. \\
\end{aligned}
\end{equation}


Note that the second-stage problem is a transportation problem and since the demand of a client is integral (0 or 1), the  optimal solution $y^S_{ij}$ is integral as well.
The special case of \eqref{urfl} where the uncertainty set contains only a single scenario corresponds to the NP-hard classical and well-studied uncapacitated facility location problem, which is hard to approximate within a constant better than 1.463 unless NP has an  $O(n^{O(\log \log n) })$-time algorithm \cite{guha1999greedy}.
We let  ({\sf LP-URFL}) denote the linear relaxation of \eqref{urfl},  where we replace $x_i \in \{0,1\}$ by $x_i \geq 0$ for each $i \in {\cal F}$.
We would like to note that, in our model, the number of scenarios could be exponential in the dimension of the problem. 
Hence, in general, even the linear relaxation of such a problem could be challenging. However, we show that  ({\sf LP-URFL})  can be solved in polynomial time using a  \textit{Static Assignment Policy} for the second-stage variables. Moreover, we can round the fractional solution losing only a constant factor, thereby  getting  a constant approximation for \eqref{urfl}. This section serves as a warm-up for introducing and motivating static assignment policies before addressing the class of capacitated robust facility location problems. 

\subsection{Static assignment policy}

Consider an optimal solution of \eqref{urfl}. Since each open facility can have an unlimited amount of supply, each client in the realized scenario is assigned to the closest facility among the opened ones.  Therefore, a client $j$ is always assigned to the same open facility in all scenarios $S$ where $j \in S$.   The same observation holds as well for ({\sf LP-URFL}) where each client is assigned to the same fractionally opened facilities independent of the realized scenario. Thus, the assignment of a client is static. This can be captured by the following policy.

\vspace{2mm}
\noindent
{\bf Static assignment policy.} There exists $y_{ij} \geq 0$ for each $ i \in {\cal F}, j \in {\cal C}$ such that
\begin{equation}
 \label{eq:policy}
 \forall {S} \in {\cal C}^k,  \forall i \in {\cal F},\forall j \in S, \qquad \qquad  y_{ij}^S=    y_{ij}. 
 \end{equation}


\noindent
{{\em Proof of Theorem \ref{thm:static-opt}}}. Let $(\mb x^*, \mb y^{*S}, S \in {\cal C}^k )$ be an optimal solution to  {\sf (LP-URFL)}. Since there are no capacities on facilities, each client $j$ is assigned to the closest fractionally opened facilities. In particular, for each $j \in {\cal C}$,
let $\pi_j$ be a permutation of ${\cal F}=\{1,\ldots,n\}$ such that  $ d_{\pi_j(1)j} \leq d_{\pi_j(2)j}  \leq \ldots \leq d_{\pi_j(n)j}$, and let
 $\ell = \min \{  p  \;   \vert  \;    x^*_{\pi_j(1)}+x^*_{\pi_j(2)}+ \ldots + x^*_{\pi_j(p)} \geq  1 \} $.
Denote $ \hat{x}_{\pi_j(\ell)} = 1- (x^*_{\pi_j(1)}+\ldots + x^*_{\pi_j({\ell-1)}}).$
The optimal solution can be written in the form \eqref{eq:policy} as follows: for $S \in {\cal C}^k$ and $j \in S$; ${y}_{ij}^S =   x^*_i$  for $i  \in \{\pi_j(1),\ldots,\pi_j(\ell-1) \}$, ${y}_{ ij}^S=  \hat{x}_i $  for $i  = \pi_j(\ell)$, and  ${y}_{ ij}^S=0 $, otherwise. 
\qed
Let ({\sf Static-URFL}) denote the  problem after  restricting the second-stage variables $y_{ij}^S$ in {\sf (LP-URFL)} to a policy  \eqref{eq:policy}, which can then be reformulated as follows:
\begin{equation} \tag{{\sf Static-URFL}}
\begin{aligned}
\min \; &  \sum_{i \in {\cal F}} c_i x_i  +\max_{ S \in {\cal C}^k} \sum_{i \in {\cal F}} \sum_{ j \in {\cal C}} {\bf 1} (j \in S) \cdot d_{ij} y_{ij} \\ 
 \text{s.t.    }\; & \sum_{i \in {\cal F}} y_{ij} \geq 1, \qquad & \forall j \in {\cal C}, \\
& x_i \geq y_{ij} \geq 0, \qquad  & \forall i  \in {\cal F},  \forall j \in {\cal C}. \\
\end{aligned}
\label{static-urfl}
\end{equation}

From Theorem \ref{thm:static-opt}, \eqref{static-urfl} is equivalent to {\sf (LP-URFL)}. The number of variables in \eqref{static-urfl} is reduced to a polynomial number since the $y_{ij}$ no longer depend  on the scenario $S$. The inner maximization problem is still taken over an exponential number of scenarios; however we can separate efficiently over these scenarios and write an efficient compact LP formulation for \eqref{static-urfl}: 
\begin{equation} \label{eq:dual}
\begin{aligned}
 \max_{S \in {\cal C}^k } \; \{ \sum_{i \in {\cal F}} \sum_{j \in  {\cal C}} {\bf 1} (j \in S) \cdot d_{ij} y_{ij}    \}   
    = \max_{\mb h \in [0,1]^{\vert {\cal C} \vert}} \; \{ \sum_{i \in {\cal F}} \sum_{j \in {\cal C}}  d_{ij} y_{ij} h_j  \;  \vert \;  \sum_{j \in {\cal C} } h_j \leq k \} \\
 =  \min_{\mu, \mb \omega \geq 0 } \; \{ k \mu + \sum_{j \in {\cal C}} \omega_j  \; \vert  \; \mu + \omega_j \geq  \sum_{i \in {\cal F}} d_{ij} y_{ij}, \; \forall j \in {\cal C} \},
\end{aligned}
\end{equation}
where the first equality holds because the optimal solution of the right maximization problem occurs at the extreme points of the $k$-ones polytope, which corresponds to the worst-case scenarios of ${\cal C}^k$ and the second equality follows from strong duality. Therefore, by dropping the min and introducing $\mu$ and all $  \omega_j$ as variables, we  reformulate \eqref{static-urfl} as the following  linear program:

\begin{equation} \label{eq:lp1} 
\begin{aligned}  
\min \; &  \sum_{i \in {\cal F}} c_i x_i + k \mu + \sum_{j \in {\cal C}} \omega_j \\
 \text{s.t.    }\; & \mu + \omega_j \geq  \sum_{i \in {\cal F}} d_{ij} y_{ij}, \qquad & \forall j \in {\cal C}, \\
 &  \sum_{i \in {\cal F}} y_{ij}\geq 1, \qquad & \forall j \in {\cal C},     \\
& x_i \geq y_{ij},    \qquad  & \forall i  \in {\cal F},  \forall j \in {\cal C}, \\
& x_i \geq 0, \; y_{ij} \geq 0, \;  \omega_j \geq 0, \; \mu \geq 0,   \qquad  & \forall i  \in {\cal F},  \forall j \in {\cal C}. 
\end{aligned}
\end{equation}


Finally, we  round  the solution of \eqref{static-urfl}   to an integral solution for \eqref{urfl} while losing  only a constant factor. This can be done using prior work on rounding techniques from the literature of deterministic facility location problems. In fact,  the LP rounding technique in Shmoys et al. \cite{shmoys1997approximation}, which gives a $4$-approximation algorithm to the deterministic uncapacitated problem also gives a $4$-approximation algorithm for \eqref{static-urfl}. The idea is to define a ball around each client of radius equal to the fractional assignment cost of the client (which is independent of any scenario for our static policy). Then, we open facilities in non-intersecting balls of ascending radius. The result is given in the following theorem and for completeness,  the details  of the rounding are in Appendix  \ref{appendix:A}.
\begin{theorem} \label{thm:rounding:uncapacitated}
\eqref{static-urfl}  can be rounded to give a 4-approximation to {\sf (URFL)}. 
\end{theorem} 
We would like to note that the focus in this section is not about finding the  best constant approximation  for {\sf (URFL)}, but we introduce it as a warm-up for motivating the static assignment policy before presenting our main result in the next section.

\section{Soft-capacitated robust facility location }

\subsection{Problem formulation}

In this section, we consider  the {\textit {soft-capacitated robust facility location}} ({\sf SCRFL}) which is similar to \eqref{urfl} except that each facility $i$ incurs a  linear supply cost,  where $c_i$ is the cost per unit of supply. We refer to $x_i$ as the supply (or capacity) in facility $i$. Each client in the realized scenario needs to be satisfied by one unit of supply. The problem is called soft-capacitated since there is no upper bound on  $x_i$, ($x_i \in \mathbb{Z_+})$. The problem is given by the following integer program:
\begin{equation}\label{scrfl} \tag{\sf SCRFL}
\begin{aligned}
\min \; &  \sum_{i \in {\cal F}} c_i x_i +\max_{ S \in {\cal C}^k} \sum_{i \in {\cal F}} \sum_{ j \in {S}}  d_{ij} y_{ij}^S &\\
\text{s.t.    } & \sum_{i \in {\cal F}} y_{ij}^S \geq 1, \qquad & \forall S \in {\cal C}^k , \forall j \in {S},     \\
& x_i \geq \sum_{j \in S} y_{ij}^S,\qquad  & \forall i  \in {\cal F},   \forall S \in {\cal C}^k, \forall j \in {S}, \\
& x_i \in \mathbb{N}, \; y_{ij}^S \geq  0, \qquad & \forall i  \in {\cal F}, \forall S \in {\cal C}^k, \forall j \in {S} . \\
\end{aligned}
\end{equation}

We let  ({\sf LP-SCRFL}) denote the linear relaxation of \eqref{scrfl},  where we replace $x_i \in \mathbb{N}$ by $x_i \geq 0$, for each $i \in {\cal F}$. We would like to note that even the linear relaxation ({\sf LP-SCRFL}) is challenging to solve since it has exponentially-many variables (scenarios).  Unlike the uncapacitated case, the static assignment policy \eqref{eq:policy} is not  optimal for ({\sf LP-SCRFL}) and the optimal assignment for each client depends, in general, on the realized scenario. In particular, the same client could be assigned to  different facilities in  different scenarios. In contrast, we show  the surprising result that a static assignment policy gives $O( \log k/ \log \log k)$-approximation  to ({\sf LP-SCRFL}). Moreover,  we can round the solution of the static assignment policy to an integral solution for \eqref{scrfl} and  only lose an additional constant factor. 
We  let ({\sf Static-SCRFL}) denote the problem when we  restrict the second-stage variables $y_{ij}^S$ in ({\sf LP-SCRFL}) to static assignment policies  \eqref{eq:policy}. The problem  can then be reformulated as follows:

\begin{equation}\label{static-scrfl} \tag{\sf Static-SCRFL}
\begin{aligned}
\min \; &  \sum_{i \in {\cal F}} c_i x_i +\max_{ S \in {\cal C}^k} \sum_{i \in {\cal F}} \sum_{ j \in {\cal C}} {\bf 1} (j \in S) \cdot  d_{ij} y_{ij} &\\
\text{s.t.    } & \sum_{i \in {\cal F}} y_{ij} \geq 1, \qquad & \forall j \in {\cal C},  \\
& x_i \geq  \max_{ S \in {\cal C}^k}     \; \;  \sum_{ j \in  {\cal C}} {\bf 1} (j \in S)    \cdot y_{ij}, \qquad & \forall i \in {\cal F}, \\
& x_i \geq 0, \; y_{ij} \geq  0, \qquad & \forall i  \in {\cal F},  \forall j \in {\cal C}. \\
\end{aligned}
\end{equation}


\subsection{An $O ( \frac{ \log k} { \log \log k} )$-approximation algorithm}

Our main contribution in this section is to show that a static assignment policy \eqref{eq:policy}  gives $O(\log k / \log \log k)$-approximation  for {\sf (LP-SCRFL)} (Theorem~\ref{thm:static:log}).
To prove this theorem, we consider an optimal solution of ({\sf LP-SCRFL}) and massage it to construct a solution of the form  \eqref{eq:policy} while losing  $O( \log k / \log \log k)$ factor. We first present our construction and several structural lemmas and then give the proof of Theorem \ref{thm:static:log}.

\vspace{2mm}
\noindent
{\bf Our construction.}
Let $\mb {x^*}: (x^*_i)_{ i \in {\cal F}}$ be an optimal first-stage solution of ({\sf LP-SCRFL}), let $\opt_1$ be the corresponding optimal first-stage cost and let $\opt_2$ be the corresponding optimal second-stage cost. We will classify the clients ${\cal C}$ into three subsets $C_1,C_2,C_3$ using Procedure \ref{alg-soft} (below) and then specify a static assignment policy for each subset. We use the following notation in the procedure.
Let  $\alpha > 1$ and $r= 5 \cdot \opt_2/ k$. For $\ell \geq 1$ and $j \in {\cal C}$, we let $B_j^{\ell}$ denote the ball centered at client $j$ of radius $\ell r$. We initialize the sets $F \leftarrow {\cal F}$ and $C \leftarrow {\cal C}$ and update them at iteration, as explained in the procedure, until $C$ becomes empty. We let $Cl(B)$ denote the set of clients in $C$ that are inside the ball $B$ and let $Sp(B)$ denote  the total optimal supply of facilities $F$ that are inside the ball $B$, i.e.,
$$ Sp(B) = \sum_{i \in  { F}} {\bf 1} (i \in B) \cdot x_i^* \qquad  \text{and} \qquad Cl(B) =  \{ j \in  { C} \; \vert \; j \in   B \}. $$
Note that both $ Sp(B)$ and $Cl(B) $ depend on the current sets of facilities $F$ and clients $C$, which we update at each iteration of the while loop in the procedure. But for the ease of notation, we do not refer to  them  with the indices $F$ and $C$. 

In the procedure, while the set $C$ is not empty, we pick a client $j \in C$ and grow three balls around it: $B_j^{2 \ell - 1}$ ({\em internal} ball), $B_j^{2 \ell }$ ({\em medium} ball) and $B_j^{2 \ell + 1}$ ({\em external} ball) starting with $\ell=1$. For each $\ell$, we check if the number of clients in the internal ball $B_j^{2 \ell - 1}$ is greater than $k$ (line 4);  if this is the case, we remove them from $C$, put them in $C_1$ and restart in line 2. If not, we check if the supply in the medium ball $B_j^{2 \ell}$ is  sufficient to satisfy half of the clients in the internal ball $B_j^{2 \ell - 1}$ (line 7); if that is not the case, we remove those clients from $C$, put them in $C_2$, and restart in line 2. Otherwise, we finally check if the supply in the the medium ball $B_j^{2 \ell}$ is sufficient to satisfy a fraction $1/2 \alpha$ of the clients in the external ball  $B_j^{2 \ell+1}$ (line 10); if that is the case we remove all the clients in $B_j^{2 \ell+1}$ and put them in $C_3$, we also  remove all the facilities in $B_j^{2 \ell}$ and restart in line 2. If none of these three conditions holds, we increase $\ell$ to $\ell+1$. First, we show that after at most $\log_{\alpha} k$ increments (i.e., $\ell \leq \log_{\alpha}k $), one of three conditions must hold and therefore we will remove some clients from $C$ and restart in line 2. Which implies that after a finite number of iterations, the set $C$ becomes empty. In particular, We have the following lemma.




\floatname{algorithm}{Procedure}

\begin{algorithm*}[t]
\caption{}\label{alg-soft}
\begin{algorithmic}[1]
\State  Initialize  $C \leftarrow {\cal C}, C_1 \leftarrow \emptyset, C_2 \leftarrow \emptyset, C_3 \leftarrow \emptyset, F \leftarrow {\cal F} $.
\While{$ { C} \neq \emptyset$ }
\State Pick a client $j \in C$. Initialize $\ell=1$
\If {  $  \vert   Cl(B_j^{2 \ell -1}) \vert \geq k$    }
\State $C_1  \leftarrow C_1 \cup Cl(B_j^{2 \ell -1})$, \; $ { C} \leftarrow  { C} \setminus Cl(B_j^{2 \ell -1})$
\State Stop, return to line 2
\EndIf
\If {  $  Sp(B_j^{2 \ell })         <  \frac{1}{2}  \cdot \vert   Cl(B_j^{2 \ell - 1})  \vert $    }
\State $C_2 \leftarrow C_2 \cup Cl(B_j^{2 \ell - 1})$, \;  $ { C} \leftarrow   {C} \setminus Cl(B_j^{2 \ell - 1})$
\State Stop, return to line 2
\EndIf
\If {  $  Sp(B_j^{2 \ell })         \geq   \frac{1}{2 \alpha}   \cdot     \vert   Cl(B_j^{2 \ell +1})  \vert $    }
\State $C_3 \leftarrow C_3 \cup Cl(B_j^{2 \ell +1})$, \; $ { C} \leftarrow  { C} \setminus Cl(B_j^{2 \ell +1})$, 
\State$ { F} \leftarrow { F} \setminus \{ i \in { F} \; \vert \; i \in  B_j^{2 \ell } \}$
\State Stop, return to line 2
\Else \;  {$ \ell \leftarrow \ell +1$, return to line 4.} 
\EndIf
\EndWhile
\end{algorithmic}
\end{algorithm*}

\begin{lemma} \label{lem:steps}
In Procedure   \ref{alg-soft},  after a finite number of iterations, the set $C$  becomes empty and $C_1 \cup C_2 \cup C_3$ is equal to $ {\cal C}$. Moreover,   $\ell$ is always less than $  \log_{\alpha} k $.
\end{lemma}

\begin{proof}
Fix a client $j$ and let $\ell \geq 1$. If none of the three conditions (``if'' statements) holds then 
$$     \alpha \cdot \vert   Cl(B_j^{2 \ell -1}) \vert   \leq 2 \alpha \cdot Sp(B_j^{2 \ell })         <          \vert   Cl(B_j^{2 \ell +1})  \vert . $$
Therefore, the number of clients grows geometrically when we increase the radius of the balls and by induction, we have that
$$ \alpha^{\ell}  \leq  \alpha^{\ell} \cdot \vert   Cl(B_j^{1}) \vert   <         \vert   Cl(B_j^{2 \ell +1}) \vert, $$
where $\vert   Cl(B_j^{1}) \vert \geq 1$, since  $Cl(B_j^{1})$  contains at least the client $j$.
Hence, after at most $\log_{    \alpha} k$ increments, we will reach $k$ clients,  and must stop by the first condition, and return to line 2. Hence, we always have $\ell \leq \log_{\alpha} k $. Finally since, we remove at least one client at each iteration of the while loop,  the set $C$ becomes empty after at most $\vert {\cal C} \vert$ iterations, and finally $C_1 \cup C_2 \cup C_3= {\cal C}$.
\qed
\end{proof}

Now we are ready to present our static assignment policy for {\sf(LP-SCRFL)}. The following three lemmas show our constructed static assignment for each client in the three subsets $C_1,C_2,C_3$. Moreover, we specify the supply used to satisfy each subset of these clients and present the analysis for the  assignment cost. 

For each client in the set $C_1$, we know that it belongs to a ball with at least $k$ clients. By the feasibility of the optimal solution, this implies that there exists $k$ units of supply in $\mb x^*$ ``close'' to this ball. Hence, we  satisfy this client by using the same fraction $1/k$ of  $\mb x^*$ (static assignment)  while paying a small assignment cost (roughly a constant times the radius of the ball). Since, there are at most $k$ clients in each scenario, and each one is using at most $\mb x^*/k$, we  need to dedicate only one $\mb x^*$ for all clients in $C_1$.  Formally, we have the following lemma.

\begin{lemma} There exists a static assignment policy for $C_1$ such that each client in $C_1$ is using at most the supply $\mb {x^*}/k$  and has an assignment cost less than  $O(\log_{\alpha}k/k ) \cdot \opt_2 $, i.e., there exists $(\tilde{y}_{ij})_{i \in {\cal F}, j \in { C_1}}$ such that  for each $j \in C_1$ :
$$         \sum_{i \in {\cal F}} \tilde{y}_{ij} \geq 1,  \qquad   \frac{x_i^*}{k}   \geq  \tilde{y}_{ij} \geq 0 , \; \forall i \in {\cal F},
\; \; \; \;\text{and} \; \; \; \; \sum_{i \in {\cal F}} d_{ij} \tilde{y}_{ij} = O(\log_{\alpha}k ) \cdot \frac{\opt_2}{k}.$$
\label{lem:C1}
\end{lemma}

\begin{proof}
Let $j$ be a client of $C_1$. It is sufficient to show that the following minimization problem is feasible and its optimal cost is  $O(\log_{\alpha}k/k ) \cdot \opt_2 $. Consider

\begin{equation} \label{eq:ss}
\min \; \left\{ \sum_{i \in {\cal F}} d_{ij} y_{ij} \; \bigg\vert \;        \sum_{i \in {\cal F}} y_{ij}  \geq  1, \; \; \frac{x_i^*}{k}       \geq  y_{ij} \geq 0 ,\; \; \forall i \in {\cal F} \right\} .
\end{equation}
Problem \eqref{eq:ss} must be feasible since the total supply in $\mb x^*$ is greater than the total demand in any scenario, i.e., $ \sum_{i \in {\cal F}} x_i^* \geq k$.
Recall that a client $j$ in $C_1$  belongs to one of the sets $Cl(B_{t}^{2 \ell - 1})$ for some  $t \in{\cal C}$  and $\ell \leq \log_{\alpha} k $ (Lemma \ref{lem:steps}) such that $\vert Cl(B_{t}^{2 \ell - 1}) \vert \geq k$. Consider a scenario $S$ formed by $k$ clients from $Cl(B_{t}^{2 \ell - 1})$. Let denote $\mb y^S$ the assignment of scenario $S$ in the optimal solution. Consider the following candidate solution for \eqref{eq:ss}:

$$ y_{ij}= \frac{1}{k} \cdot \sum_{ p \in S} y^S_{ip}, \; \forall i \in{\cal F}.$$
We have, by the feasibility of the optimal solution, for each $i \in {\cal F}$,
$ 0 \leq y_{ij} \leq  \frac{1}{k} x_i^* $
and
$$\sum_{i \in {\cal F}} y_{ij}  =   \frac{1}{k} \cdot       \sum_{i \in {\cal F}} \sum_{ p \in S} y^S_{ip}  \geq   \frac{1}{k} \sum_{ p \in S} 1 = 1.   $$
Therefore, our solution is feasible for \eqref{eq:ss}. Moreover,  we have
\begin{align*}
 \sum_{i \in {\cal F}} d_{ij}  y_{ij}  &= \frac{1}{k} \cdot \sum_{i \in {\cal F}} \sum_{ p \in S} d_{ij} y^S_{ip} \\
 & \leq   \frac{1}{k} \cdot  \sum_{i \in {\cal F}}  \sum_{ p \in S}  d_{ip} y^S_{ip}+ \frac{1}{k} \cdot \sum_{ p \in S}  d_{pj}         \\
 & \leq \frac{1}{k} \cdot \opt_2   + 2(2\ell -1) r \\
 & \leq \frac{\opt_2}{k}  + 2(2 \log_{\alpha}k      -1) \cdot 5 \cdot \frac{\opt_2}{k}  = O(\log_{\alpha}k ) \cdot \frac{\opt_2}{k},
\end{align*}
where the first inequality follows from the triangle inequality and the fact that 
$\sum_{i \in {\cal F}}  y_{ip}^S=1$ for all $p \in S$ in the optimal solution. For the second inequality, we use the definition of $\opt_2$ to bound the first term,  the second term $d_{pj}$ is bounded by the diameter of the ball $B_j^{2 \ell -1}$ which contains client $j$ and  all clients $p \in S$.
\qed
\end{proof}

Now, consider the set $C_2$. By construction, these are clients such that there is not  enough  supply  within a distance $r= 5\opt_2/k$ to satisfy half of them. Therefore, intuitively they need to pay ``large'' distances  in the optimal assignment cost if they show up all together in the same scenario. In the following lemma, we show that we can  have no more than $k$ of these clients. As we would show later, this would imply that we can dedicate a supply $\mb x^*$ to $C_2$ and make a static assignment of all the clients $C_2$ to this $\mb x^*$. 
\begin{lemma} The set $C_2$ has at most  $k$ clients. 
\label{lem:C2}
\end{lemma}

\begin{proof}
Suppose, for the sake of contradiction, that $\vert C_2 \vert >  k$. Let $ G_1, G_2, \ldots, G_T$ be the disjoint  subsets of clients added at each iteration in the construction of $C_2$ in Procedure \ref{alg-soft}. In particular,  $C_2= G_1 \cup G_2 \cup \ldots \cup G_T$ for some $T$ where:
\begin{itemize}
    \item [(i)] for $t=1,2,\ldots,T, $ $G_t=Cl(B_{j_t}^{2 \ell_t-1})$ for some  client $j_t$ and  $  1  \leq \ell_t  \leq \log_{\alpha} k $.
    \item  [(ii)] the supply $Sp(B_{j_t}^{2 \ell_t})$ is less than half of the clients in $G_t$.
    \item [(iii)] each set $G_t$ has strictly less than $k$ clients, since the procedure has to fail the first ``if'' statement before adding $G_t$ into $C_2$. 
\end{itemize}
Recall that $$ Sp(B_{j_t}^{2 \ell_t }) = \sum_{i \in  { F}} {\bf 1} (i \in B_{j_t}^{2 \ell_t }) \cdot x_i^*, $$ 
where $F$ is the current set of facilities in the procedure (and is not all $\cal F$ since some facilities have been removed in line 12 of the procedure). However, we would like to emphasize that when a  facility has been removed (in line 12 of the procedure),  all clients within distance  $r$ from this facility were removed as well (line 11). This is true, since  when we remove the facilities in a medium ball, say $B_j^{2 \ell }$, (line 12), we remove all  clients in the corresponding external ball $B_j^{2 \ell+1 }$ (line 11). Hence, the remaining clients in $C$ are at least $(2\ell+1)r - ( 2 \ell)r =r$ away from the removed facilities. In particular, for the clients $G_t$, the supply that has been removed before they were added to $C_2$ is at least $r$ away from them. Therefore, all the facility in $\cal F$ that are within a distance $r$ from a client in $G_t$ belong to the set $F$ that verifies $\sum_{i \in  { F}} {\bf 1} (i \in B_{j_t}^{2 \ell_t }) \cdot x_i^* \leq 2 \cdot \vert G_t \vert $. This implies that the  supply of all facilities within a distance $r$ from $G_t$ in the optimal solution, is less than half of the clients $G_t$. Hence, if all of the clients $G_t$ show up in a scenario,  the optimal second-stage solution needs to pay an assignment cost of at least $r \cdot \vert G_t \vert /2$.

Order the sets $G_t$ according to their cardinalities: wlog assume that $\vert G_1 \vert \geq \vert G_2 \vert \geq \ldots \geq \vert G_T \vert$.  
We construct a scenario $\hat{S}$ by taking clients from the sets $G_1$, $G_2, \ldots$ until we hit $k$. This is possible since by assumption $\vert C_2 \vert >  k$. Assume that
$$ \vert G_1 \vert + \vert G_2 \vert + \ldots \vert G_{p-1}  \vert    +  \vert \bar{G}_{p}  \vert = k ,$$
for some $p$, where $ 2 \leq p \leq T$. Note that $\bar{G}_{p}$ is a subset of ${G}_{p}$, since we can reach $k$ before taking all the clients of the last set ${G}_{p}$.  
For each $t=1, \ldots, p-1$, the optimal second-stage decision needs to pay at least $r \cdot \vert G_t \vert /2$. Therefore,
 

  $$ \opt_2 \geq  \frac{1}{2} \cdot  r \cdot ( \vert G_1 \vert + \vert G_2 \vert + \ldots \vert G_{p-1}  \vert   ) .$$
We did not include  $  G_{p} $, since not all these clients  are necessary in the scenario $\hat{S}$, but only $\vert \bar{G}_{p}  \vert$ of them. Since $\bar{G}_{p}$ has the smallest cardinality
$$  \vert G_1 \vert + \vert G_2 \vert + \ldots \vert G_{p-1}  \vert \geq     \frac{1}{2} ( \vert G_1 \vert + \vert G_2 \vert + \ldots \vert G_{p-1}  \vert    +  \vert \bar{G}_{p}  \vert )  = \frac{k}{2}. $$
Therefore,
$$  \opt_2 \geq      r \cdot \frac{k}{4}   =   5 \cdot \frac{\opt_2}{4}  ,   $$
  which is a contradiction.
 Therefore, $\vert C_2 \vert \leq k$. 
 \qed
\end{proof}

Finally, for clients $C_3$, we show that there exists $ \vert C_3 \vert /2\alpha$ units of supply ``close'' to them. In particular, we can  multiply these units by $2 \alpha$, dedicate them to $C_3$ and make a static assignment for $C_3$. We have the following lemma.

\begin{lemma}
There exists a supply $\mb {\hat{x}}$  that has a cost at most  $2 \alpha \cdot \opt_1$ and there exists a static assignment policy such that all clients in $C_3$  are assigned to  supply $\mb {\hat{x}}$ and each client in $C_3$ has an assignment cost that is  $O(\log_{\alpha}k /k ) \cdot \opt_2$, i.e., there exists $(\hat{y}_{ij})_{i \in {\cal F}, j \in { C_3}}$ and $(\hat{x}_i)_{i \in{\cal F}}$
 such that  for all $j \in C_3$ :
$$         \sum_{i \in {\cal F}} \hat{y}_{ij} \geq 1,   \qquad  \hat{x}_i  \geq \sum_{j \in C_3} \hat{y}_{ij},\; \forall i \in {\cal F},   \qquad    \hat{x}_i \geq 0, \; \hat{y}_{ij} \geq 0 , \; \forall i \in {\cal F},$$
$$\sum_{i \in {\cal F}} c_i   \hat{x}_i \leq 2 \alpha  \cdot \opt_1 \qquad 
\text{and} \qquad  \sum_{i \in {\cal F}} d_{ij} \hat{y}_{ij} = O(\log_{\alpha}k ) \cdot \frac{\opt_2}{k}.
$$
 \label{lem:C3}
\end{lemma}

\begin{proof} 
Let $ G_1, G_2, \ldots, G_T$ be the disjoint  subsets of clients added at each iteration to  construct  $C_3$ in Procedure \ref{alg-soft}. In particular,  $C_3= G_1 \cup G_2 \cup \ldots \cup G_T$ for some $T$ such that:  for all $t=1,2,\ldots,T, $ $G_t=Cl(B_{j_t}^{2 \ell_t+1})$ for some  client $j_t$ and some integer $\ell_t$ with $  1  \leq \ell_t  \leq \log_{\alpha} k $.  Moreover, the supply $Sp(B_{j_t}^{2 \ell_t})$ is greater  than a $1/2\alpha$ fraction of the clients in $G_t$. Hence, for each ball $B_{j_t}^{2 \ell_t}$, we multiply the supply by $2 \alpha $, move it to the cheapest facility in this ball and make a static assignment of all clients  $G_t$ to this cheapest facility. Since the supply in $B_{j_t}^{2 \ell_t}$ is removed along with clients $G_t$, it will not be used by the other clients in $C_3$.

Formally, let $i_t$ be the cheapest facility in the ball $B_{j_t}^{2 \ell_t}$. We define, for each facility $i$ in $B_{j_t}^{2 \ell_t}$, $\hat{x}_{i} = 2 \alpha \sum_{i' \in  { F}} {\bf 1} (i' \in B_{j_t}^{2 \ell_t}) \cdot x_{i'}^* $ if $i =i_t$ and $\hat{x}_{i}=0$ otherwise. For each client $j \in C_3$, we let $\hat{y}_{ij}=1$ for $i=i_t$ and $j \in G_t$, and let $\hat{y}_{ij}=0$, otherwise. Therefore, the first desired constraints in the lemma are verified. Let us check the last one. The distance between a client and its assigned facility in our solution is at most $r$ plus the diameter of the ball $B_{j_t}^{2 \ell_t}$, i.e., $$r+ 4 {\ell}_t r \leq (4 \log_{\alpha} k +1) \cdot 5 \cdot \opt_2/k = O( \log_{\alpha}k/k) \cdot \opt_2.$$
\qed
\end{proof}


\noindent
{\em Proof of Theorem \ref{thm:static:log}.}
Let $(\tilde{y}_{ij})_{i \in {\cal F}, j \in { C_1}}$ be the solution given in Lemma \ref{lem:C1} for satisfying the clients in $C_1$. We dedicate a supply $\mb x^*$ to clients $C_1$. Let $(\hat{y}_{ij})_{i \in {\cal F}, j \in { C_3}}$ and $(\hat{x}_i)_{i \in{\cal F}}$ be the solution given in Lemma \ref{lem:C3} for satisfying the clients in $C_3$. Finally, we know from Lemma \ref{lem:C2} that $C_2$ has at most $k$ clients, and therefore $C_2$ is a scenario. So we dedicate a supply $\mb x^*$ to $C_2$ and let the optimal assignment $y_{ij}^{C_2}$ be  our static assignment solution for $C_2$. In particular, we give the following solution to {\sf (LP-SCRFL)}, where the first stage solution is $  2 \mb x^* +   \mb {\hat{x}} $ and the static assignment policy is for all $ i \in {\cal F}$: $y_{ij} = \tilde{y}_{ij} $  for $ j \in C_1 $, $ y_{ij}= {y}^{C_2}_{ij}$ for $j \in C_2 $, and  $ y_{ij}=\hat{y}_{ij} $ for $ j \in C_3$. It is clear that $ \sum_{i \in {\cal F}} y_{ij} \geq 1 $, for each $j$ in $C_1 \cup C_2 \cup C_3$. Moreover, for any scenario $S \in {\cal C }^k$ and $i \in {\cal F}$,
\begin{align*}
    \sum_{j \in S} y_{ij} &=  \sum_{j \in S \cap C_1} \tilde{y}_{ij} + \sum_{j \in S \cap C_2} y^{C_2}_{ij} + \sum_{j \in S \cap C_3} \hat{y}_{ij} \\
    &\leq  \sum_{j \in S\cap C_1} \frac{x_i^*}{k} + x_i^* + \hat{x}_i \leq  2x_i^* + \hat{x}_i. 
\end{align*}
Therefore, our solution is feasible for {\sf (LP-SCRFL)}. Let us evaluate its cost. The cost of the first stage is at most  $2 OPT_1 + 2 \alpha \opt_1= O(\alpha) \cdot\opt_1$. 
For the second-stage cost, consider any scenario $S \in {\cal C}^k$, 
We have
\begin{align*}
    \sum_{i \in {\cal F}} \sum_{ j \in S} d_{ij} y_{ij} &=  \sum_{j \in S \cap C_1}  \sum_{i \in {\cal F}} d_{ij}  \tilde{y}_{ij} +  \sum_{j \in S \cap C_2}   \sum_{i \in {\cal F}} d_{ij}y^{C_2}_{ij} + \sum_{j \in S \cap C_3} \sum_{i \in {\cal F}} d_{ij}\hat{y}_{ij} \\
    &\leq  \sum_{j \in S\cap C_1} O(\log_{\alpha}k ) \cdot \frac{\opt_2}{k} + \opt_2 + \sum_{j \in S \cap C_3} O(\log_{\alpha}k ) \cdot \frac{\opt_2}{k} \\
    &\leq  O(\log_{\alpha}k ) \cdot \opt_2 + \opt_2 +  O(\log_{\alpha}k ) \cdot \opt_2  \\
    &= O(\log_{\alpha}k ) \cdot \opt_2.
\end{align*}
By balancing the terms $\alpha $ and $\log_{\alpha} k$, we choose $ \alpha =\log k / \log \log k$ which gives $O(\log k / \log \log k)$-approximation to {\sf (LP-SCRFL)}.
\qed

Similar to the uncapacitated problem, we can solve \eqref{static-scrfl} efficiently using a compact linear program. In fact, we dualize the inner maximization problem in the objective function of \eqref{static-scrfl} in the same way as \eqref{eq:dual}. In addition to that, we reformulate the second constraint in \eqref{static-scrfl} using the same dualization technique as follows:  for each $i \in {\cal F}$,

\begin{align*}
 \max_{S \in {\cal C}^k } \; \{ \sum_{ {j \in  S}}   y_{ij}  \}   
 &   = \max_{\mb h \in [0,1]^{\vert C \vert}} \; \{ \sum_{ {j \in {\cal C}}}   y_{ij} h_j  \; \; \vert \; \; \sum_{j \in {\cal C} } h_j \leq k \} \\
 &=  \min_{\eta_i, \lambda_{ij} \geq 0 } \; \{ k \eta_i + \sum_{j \in {\cal C}} \lambda_{ij}  \; \vert  \; \eta_i + \lambda_{ij} \geq    y_{ij}, \; \forall j \in {\cal C} \}.
\end{align*}
The linear program is given by
\begin{equation}  
\begin{aligned}  
\min \; &  \sum_{i \in {\cal F}} c_i x_i + k \mu + \sum_{j \in {\cal C}} \omega_j \\
 \text{s.t.    }\; & \mu + \omega_j \geq  \sum_{i \in {\cal F}} d_{ij} y_{ij}, \qquad & \forall j \in {\cal C}, \\
 &  \sum_{i \in {\cal F}} y_{ij}\geq 1, \qquad & \forall j \in {\cal C} ,    \\
  & x_i \geq k \eta_i + \sum_{j \in {\cal C}} \lambda_{ij} ,& \forall i \in {\cal F} ,    \\
    &\eta_i +  \lambda_{ij} \geq y_{ij}, & \forall i \in {\cal F} , \forall j \in {\cal C},   \\
& x_i, y_{ij}, \lambda_{ij}, \eta_i, \omega_j, \mu \geq 0,  \qquad  & \forall i  \in {\cal F},  \forall j \in {\cal C}, \\
\end{aligned}
\end{equation}
which can be reduced, after removing the variables $y_{ij}$, to 
\begin{equation} \label{eq:lp2} 
\begin{aligned}  
\min \; &  \sum_{i \in {\cal F}} c_i x_i + k \mu + \sum_{j \in {\cal C}} \omega_j \\
 \text{s.t.    }\; & \mu + \omega_j \geq  \sum_{i \in {\cal F}} d_{ij} (\eta_i +  \lambda_{ij}) , \qquad & \forall j \in {\cal C}, \\
 &  \sum_{i \in {\cal F}} \eta_i +  \lambda_{ij}\geq 1, \qquad & \forall j \in {\cal C} ,    \\
  & x_i \geq k \eta_i + \sum_{j \in {\cal C}} \lambda_{ij}, & \forall i \in {\cal F}  ,   \\
& x_i, \lambda_{ij}, \eta_i, \omega_j, \mu \geq 0,  \qquad  & \forall i  \in {\cal F},  \forall j \in {\cal C}. \\
\end{aligned}
\end{equation}

Finally, we  round the optimal solution of \eqref{static-scrfl}  to an integral solution using the filtering and rounding techniques from Shmoys et al. \cite{shmoys1997approximation} while losing only a  factor of 12. Again, this rounding technique was designed for the deterministic facility location problem, but the same argument works as well for \eqref{static-scrfl}. Finally, since  \eqref{static-scrfl} gives $O(\log k / \log \log k)$-approximation to ({\sf LP-SCRFL}) and we only loose a constant factor in the rounding, this results in $O(\log k / \log \log k)$- approximation algorithm for ({\sf SCRFL}). We state the result in the following theorem and, for completeness, we present  the details  of the rounding in Appendix \ref{appendix:B}.
\begin{theorem} \label{thm:rounding:capacitated}
\eqref{static-scrfl} can be rounded to  give  $O(\frac{\log k}{\log \log k})$-approximation algorithm to \eqref{scrfl}. 
\end{theorem}
Note that after rounding the supply in our solution of \eqref{static-scrfl},  the integral second-stage assignment for each realized scenario is a transportation problem and therefore its optimal solution is integral. We would like to emphasize that while the fractional assignment in our solution is static,  our integral assignment is not  necessarily static. In fact, a policy with an integral static assignment could even be bad our model.

\section{Conclusion.}
In this paper, we give a $O(\log k / \log \log k)$-approximation for soft-capacitated robust facility location problems with an implicit model of demand uncertainty. It is an interesting open question to study whether there exists  a constant approximation algorithm for the problem, even in special cases such as the Euclidean metric. Our solution approach relies on static fractional assignment policies, which we show are optimal for the  uncapacitated problem and give a strong theoretical guarantee for soft-capacitated case. 
Static assignment policies, while reasonable for the case of soft-capacities can be shown to be arbitrarily bad for the case of hard-capacities, where in addition to cost per unit, there is also an upper bound on supply at each facility. It is another interesting open direction to study any non-trivial approximation in this setting.

{\small
{\newpage
\bibliographystyle{abbrv}
\bibliography{robust}}}

\begin{appendix}

\section{Rounding and proof of Theorem \ref{thm:rounding:uncapacitated}} \label{appendix:A}

\noindent
{\bf Rounding} (Shmoys et al.\cite{shmoys1997approximation}).
Recall that $n$ is the number of facilities and $m$ is the number of clients. Let $(\mb x^*, \mb y^*)= \left( (x^*_i)_{ i \in {\cal F}},  (y^*_{ij})_{ i \in {\cal F}, j \in {\cal C}} \right)$ be an optimal solution for  \eqref{static-scrfl}. Let $\opt_1$ and $\opt_2$ respectively be the corresponding optimal first-stage and  second-stage cost. In particular,  $\opt_2= \max_{S \in {\cal C}^k } \sum_{i \in {\cal F}} \sum_{ j \in {\cal C}}    {\bf 1}( j \in S) \cdot   d_{ij} y^*_{ij} $.
For each client $j \in {\cal C}$, let
$$L_j = \sum_{i \in {\cal F} }d_{ij} y^*_{ij}.$$
We sort $L_j$ in ascending order. Suppose without loss of generality, that
$L_1 \leq L_2 \leq \ldots \leq L_m$.
For each client $j \in {\cal C}$, we consider $B_j$ the  ball centered at the client $j$ with radius $\alpha L_j$ (where $\alpha \geq 1$ will be defined later).
We choose greedily the non-intersecting  balls in ascending order of $L_j$ and   open the cheapest facility in each selected ball. That is, we consider every client in this sorted order, and only include its ball if it does not intersect any previously selected ball.
Each client is assigned to its closest opened facility. For each client $j \in {\cal C}$,
$$ L_j \geq \sum_{i \notin B_j}   d_{ij} y^*_{ij}  \geq \sum_{i \notin B_j}  \alpha  L_j y^*_{ij}.$$
Hence,
$$ \frac{1}{\alpha}\geq \sum_{i \notin B_j} y^*_{ij},$$
which implies
$$ \sum_{i \in B_j}  y^*_{ij} \geq 1- \frac{1}{\alpha}. $$
Therefore,
$$ \sum_{i \in B_j} x^*_i \geq 1- \frac{1}{\alpha} .$$
By summing over the non-intersecting balls, the cost of our opened facilities is less than $\frac{1}{1-\frac{1}{\alpha}} \opt_1$.
Now let us consider the assignment cost. For client $j$,  if its ball $B_j$ is chosen by the greedy procedure above, then the client is assigned to the opened facility $i$ in $B_j$ and
$ d_{ij} \leq \alpha L_j$. If not, the client $j$ is assigned to a facility $i$ no further away than the one chosen in the ball $B_k$ that intersected $B_j$. By the triangle inequality:
 $$ d_{ij} \leq d_{iv} +  d_{jv}  \leq  \alpha L_j  + 2 \alpha L_k \leq 3 \alpha L_j ,$$
 where $ v \in B_j \cap B_k$ and $L_k \leq L_j$. Hence, for any scenario  $S \in {\cal C}^k$, the second-stage cost is at most $3 \alpha \cdot \sum_{j \in S} L_j$, which is at most $3 \alpha \cdot \opt_2$ . By optimizing over $\alpha$, we take $\alpha= \frac{4}{3}$, which results in the 4-approximation with respect to $\opt_1+ \opt_2$. Moreover, since \eqref{static-urfl} is equivalent to {\sf(LP-URFL)} and  {\sf(LP-URFL)} is a relaxation of \eqref{urfl}, the rounding above gives 4-approximation algorithm for  \eqref{urfl}.


\section{Rounding and proof of Theorem \ref{thm:rounding:capacitated}} \label{appendix:B}

\noindent
{\bf Rounding} (Shmoys et al.\cite{shmoys1997approximation}). Let $(\mb x^*, \mb y^*)= \left( (x^*_i)_{ i \in {\cal F}},  (y^*_{ij})_{ i \in {\cal F}, j \in {\cal C}} \right)$ be an optimal solution of \eqref{static-scrfl}. Let $ z= \max_{ S \in {\cal C}^k} \sum_{i \in {\cal F}, j \in {S}}  d_{ij} y^*_{ij}$ and set $ \alpha \in (0,1)$. For each client $j \in {\cal C}$, let  $\pi$ be a permutation of facilities that serve $j$ in the optimal solution of \eqref{static-scrfl} such that $ d_{\pi(1)j} \leq d_{\pi(2)j}  \leq \ldots \leq d_{\pi(n)j}$.  Define the radius  $d_j (\alpha )= d_{\pi(i^*)j}$ where $i^*= \min \{    i' \; \vert \; \sum_{i=1}^{i'} y^*_{\pi(i)j} \geq   \alpha  \}$. Following the same definitions in Shmoys el al. \cite{shmoys1997approximation}, we say that a solution $(\mb x, \mb y)$ is $g$-close if for any $j \in {\cal C}$, $y_{ij} >0$ implies that $d_{ij}  \leq g_j$.

\vspace{2mm}
\noindent
{\textit {Claim.}} Given a feasible solution $(\mb x, \mb y)$, we can construct a $g$-close feasible solution $(\bar{\mb x}, \bar{ \mb y})$ such that 
$\bar{\mb x} = \frac{1}{\alpha} \mb x$, $\bar{z} =  \frac{1}{\alpha} z$   and  $g_j \leq d_j( \alpha ), \; \forall j \in {\cal C}$ where $\bar{z}= \max_{ S \in {\cal C}^k} \sum_{i \in {\cal F}, j \in {S}}  d_{ij} \bar{y}_{ij}$.
\begin{proof}
In fact, for each $j$, we set $\bar{y}_{\pi(i)j}=y_{\pi(i)j} / \alpha      $ for $ 1 \leq i \leq i^*$ and $\bar{y}_{\pi(i)j}=0$ otherwise. We have
$$ \sum_{i \in {\cal F}} \bar{y}_{ij} =  \sum_{i=1}^{i^*}     \frac{1}{\alpha}    y_{ \pi(i)   j}   \geq 1 .$$ 
The second constraint in \eqref{static-scrfl}  is trivially verified since $x$ is multiplied by $\frac{1}{\alpha}$ and $y_{ij}$ are either multiplied by $\frac{1}{\alpha}$ or set to $0$. Moreover, all edges in our solution with a positive flow have a cost at most  $d_j(\alpha)$. Hence, $(\bar{\mb x}, \bar{ \mb y})$ is $g$-close with $g_j \leq d_j( \alpha )$ for all $j$. 
\qed
\end{proof}

Let $(\bar{\mb x}, \bar{ \mb y})$ be the $g$-close solution corresponding to the optimal solution $(\mb x^*, \mb y^*)$. For the ease of notation, we just use the notation  $(\mb x, \mb y)$ instead of $(\bar{\mb x}, \bar{ \mb y})$ in the rest of the proof. We round-up all $x_i \geq \frac{1}{2}$ to $\lceil x_i \rceil$. We pay at most a factor $2$ in the first-stage cost. Let us focus on the remaining fractional facilities, say $\hat{{\cal F}}$, i.e., 
$$ \hat{{\cal F}}= \{ i \in {\cal F} \; \; \vert \; \;   0 < x_i < 1/2 \}. $$
We let $\hat{{\cal C}}$ denote the set of clients such that half of their supply is coming from $ \hat{{\cal F}}$, i.e.,
$$ \hat{{\cal C}}= \{ j \in {\cal C} \; \; \vert \; \;   \sum_{i\in \hat{{\cal F}}} y_{ij} \geq \frac{1}{2} \} .$$
We sort the clients in $\hat{{\cal C}}$  in ascending order of $g_j$ and do the following until $\hat{{\cal C}} $ becomes empty. We take  the client $j$ with the smallest $g_j$ in $\hat{{\cal C}}$, let say $j'$. Let
$$ V= \{ i \in \hat{{\cal F}}  \; : \; y_{ij'}  > 0 \},$$ 
and
$$ T= \{ j \in {\cal C} \; : \;    \exists i \in V  \text{ s.t. }         y_{ij}  > 0 \}.$$ 
We put the  $\lceil \sum_{i \in V} x_i \rceil $ units of supply in the cheapest facility in $V$ which we denote $f_c$. We set each $x_i$ in $V \setminus \{f_c\}$ to 0. The clients in $T$ are the ones affected by this change. We will route all of their demand from $V$ to $f_c$ and make this assignment static. Note that, at this step, we have routed only their demand from $V$ and not their entire demand.
This  is feasible because 
$$ \lceil \sum_{i \in V} x_i \rceil \geq  \sum_{i \in V} x_i  \geq \sum_{j \in {\cal C}}  {\bf 1 }(j \in S) \cdot \sum_{i \in V}   y_{ij}  $$ for any scenario $S$. Since we choose the clients in $\hat{\cal C}$ in ascending order of $g_j$,  the triangle inequality ensures that the solution is $3 g$-close.  We keep doing this until $\hat{\cal C}$ becomes empty. 

Let us analyze the final cost. We lost a factor $\frac{1}{\alpha}$ in both first-stage and second-stage cost. Then, we lost a factor $2$ in the first-stage cost by rounding up the solution and moving the supply to the  cheapest facilities after each rounding. We lost another factor $2$ in the first stage to satisfy ${\cal C} \setminus \hat{\cal C}$. For the second stage, we have a solution that is  $3g$-close with $g_j \leq d_j( \alpha)$. Note that
$$ d_j (\alpha) \leq \frac{1}{1- \alpha} \sum_{i \in {\cal F}} d_{ij} y_{ij}.$$ 
In particular, for any scenario $S$,
$$ \sum_{j \in S} g_j \leq \frac{1}{1- \alpha} \sum_{j \in S}  \sum_{i \in {\cal F}} d_{ij} y_{ij}.$$ 
Hence, we loose $\frac{3}{1-\alpha}$ in the second stage. Overall, the factor is 
$$ \frac{4}{\alpha} \opt_1+ \frac{3}{\alpha( 1- \alpha)} \opt_2.$$ 
We set $\alpha=\frac{1}{2}$, which gives 12 approximation to ({\sf LP-SCRFL}). Finally, since  \eqref{static-scrfl} gives $O(\log k / \log \log k)$-approximation to {\sf(LP-SCRFL)} and {\sf(LP-SCRFL)} is a relaxation of ({\sf SCRFL}), this results in $O(\log k / \log \log k)$-approximation algorithm for ({\sf SCRFL}).

\vspace{2mm}
\noindent

\end{appendix}

\end{document}